\documentclass[11pt]{article}

\usepackage{mathpazo}   
\usepackage{comment}
\usepackage[title]{appendix}
\usepackage[final]{showkeys}
\usepackage{seqsplit}
\usepackage{xstring}

\usepackage{wrapfig}
\usepackage{subcaption}
\usepackage{tikz}
\usepackage{pgfplots}
\pgfplotsset{compat=1.18}
\usetikzlibrary{patterns, pgfplots.fillbetween}
\graphicspath{{.}}
\usepackage[font=footnotesize]{caption}
\usepackage{nameref}
\usepackage{mathtools}
\usepackage{empheq}
\usepackage{comment}
\usepackage[shortlabels,inline]{enumitem}
\setlist[enumerate]{nosep}
\usepackage[colorlinks=true,
linkcolor=refkey,
urlcolor=lblue,
citecolor=red]{hyperref}
\usepackage[doc,wmm,hhb]{optional}
\usepackage{xcolor}

\usepackage{float}
\usepackage{soul}
\usepackage{graphicx}
\definecolor{labelkey}{rgb}{0,0.08,0.45}
\definecolor{refkey}{rgb}{0,0.6,0.0}
\definecolor{Brown}{rgb}{0.45,0.0,0.05}
\definecolor{lime}{rgb}{0.00,0.8,0.0}
\definecolor{lblue}{rgb}{0.5,0.5,0.99}
\definecolor{OliveGreen}{rgb}{0,0.6,0}
\definecolor{tyrianpurple}{rgb}{0.4, 0.01, 0.24}

\colorlet{hlcyan}{cyan!30}

\usepackage{stmaryrd}
\usepackage{amssymb}

\makeatletter
\def\namedlabel#1#2{\begingroup
   \def\@currentlabel{#2}%
   \label{#1}\endgroup
}

\makeatother

\newcommand{\seppthree}{\setlength{\itemsep}{-3pt}}

\usepackage[margin=0.92in,footskip=0.25in]{geometry}
\newcommand{\weakly}{\ensuremath{\:{\rightharpoonup}\:}}

\newcommand{\nnn}{\ensuremath{{n\in{\mathbb N}}}}

\newcommand{\thalb}{\ensuremath{\tfrac{1}{2}}}

\newcommand{\fenv}[1]%
{\ensuremath{\,\overrightarrow{\operatorname{env}}_{#1}}}
\newcommand{\benv}[1]%
{\ensuremath{\,\overleftarrow{\operatorname{env}}_{#1}}}

\newcommand{\scal}[2]{\left\langle{#1},{#2}  \right\rangle}

\newcommand{\RR}{\ensuremath{\mathbb R}}

\newcommand{\RX}{\ensuremath{\,\left]-\infty,+\infty\right]}}

\newcommand{\NN}{\ensuremath{\mathbb N}}

\newcommand{\prox}{\ensuremath{\operatorname{P}}}

\newcommand{\ran}{\ensuremath{{\operatorname{ran}}\,}}

\newcommand{\cspan}{\ensuremath{\overline{\operatorname{span}}\,}}

\newcommand{\Par}{\ensuremath{\operatorname{par}}}

\newcommand{\Argmin}{\ensuremath{\operatorname{Argmin}}\,}

\newcommand{\Fix}{\ensuremath{\operatorname{Fix}}}
\newcommand{\Id}{\ensuremath{\operatorname{Id}}}

\newcommand{\Diag}{\ensuremath{\operatorname{Diag}}}

\newcommand{\minimize}[2]{\ensuremath{\underset{\substack{{#1}}}{\mathrm{minimize}}\;\;#2 }}

{\begin{list}{}{%
\settowidth{\labelwidth}{\textrm{#1~}}%
\setlength{\leftmargin}{\labelwidth+\labelsep}}}
{\end{list}}
\usepackage{amsthm}
\usepackage{aliascnt}
\makeatletter
\def\th@plain{%
	\thm@notefont{}
	\itshape 
}
\def\th@definition{%
	\thm@notefont{}
	\normalfont 
}
\makeatother
\usepackage[capitalize,nameinlink]{cleveref}
\crefname{equation}{}{equations}

\crefname{item}{}{items}
\crefname{enumi}{}{}
\newtheorem{theorem}{Theorem}[section]

\newaliascnt{lemma}{theorem}
\newtheorem{lemma}[lemma]{Lemma}
\aliascntresetthe{lemma}
\crefname{lemma}{Lemma}{Lemmas}
\Crefname{lemma}{Lemma}{Lemmas}

\newaliascnt{lem}{theorem}

\aliascntresetthe{lem}
\crefname{lem}{Lemma}{Lemmas}
\Crefname{lem}{Lemma}{Lemmas}

\newaliascnt{corollary}{theorem}
\newtheorem{corollary}[corollary]{Corollary}
\aliascntresetthe{corollary}
\crefname{corollary}{Corollary}{Corollaries}
\Crefname{corollary}{Corollary}{Corollaries}

\newaliascnt{cor}{theorem}

\aliascntresetthe{cor}
\crefname{cor}{Corollary}{Corollaries}
\Crefname{cor}{Corollary}{Corollaries}

\newaliascnt{proposition}{theorem}
\newtheorem{proposition}[proposition]{Proposition}
\aliascntresetthe{proposition}
\crefname{proposition}{Proposition}{Propositions}
\Crefname{proposition}{Proposition}{Propositions}

\newaliascnt{prop}{theorem}

\aliascntresetthe{prop}
\crefname{prop}{Proposition}{Propositions}
\Crefname{prop}{Proposition}{Propositions}

\newaliascnt{definition}{theorem}

\aliascntresetthe{definition}
\crefname{definition}{Definition}{Definitions}
\Crefname{definition}{Definition}{Definitions}

\newaliascnt{defn}{theorem}

\aliascntresetthe{defn}
\crefname{defn}{Definition}{Definitions}
\Crefname{defn}{Definition}{Definitions}

\newaliascnt{thm}{theorem}

\aliascntresetthe{thm}
\crefname{thm}{Theorem}{Theorems}
\Crefname{thm}{Theorem}{Theorems}

\newaliascnt{example}{theorem}
\newtheorem{example}[example]{Example}
\aliascntresetthe{example}
\crefname{example}{Example}{Examples}
\Crefname{example}{Example}{Examples}

\newaliascnt{ex}{theorem}
\newtheorem{ex}[ex]{Example}
\aliascntresetthe{ex}
\crefname{ex}{Example}{Examples}
\Crefname{ex}{Example}{Examples}

\newaliascnt{fact}{theorem}
\newtheorem{fact}[fact]{Fact}
\aliascntresetthe{fact}
\crefname{fact}{Fact}{Facts}
\Crefname{fact}{Fact}{Facts}

\newaliascnt{remark}{theorem}
\newtheorem{remark}[remark]{Remark}
\aliascntresetthe{remark}
\crefname{remark}{Remark}{Remarks}
\Crefname{remark}{Remark}{Remarks}

\newaliascnt{rem}{theorem}

\aliascntresetthe{rem}
\crefname{rem}{Remark}{Remarks}
\Crefname{rem}{Remark}{Remarks}

\crefname{chapter}{Appendix}{chapters}

\providecommand{\RR}{\mathbb{R}}

\providecommand{\ran}{\operatorname{ran}}

\newcommand{\fix}{\ensuremath{\operatorname{Fix}}}

\providecommand{\Id}{\operatorname{{ Id}}}

\providecommand{\NN}{\mathbb{N}}

\providecommand{\fix}{\operatorname{Fix}}
\providecommand{\ran}{\operatorname{ran}}

\providecommand{\Id}{\operatorname{Id}}

\newcommand{\cran}{\ensuremath{\overline{\operatorname{ran}}\,}}

\providecommand{\RR}{\mathbb{R}}
\providecommand{\NN}{\mathbb{N}}

\definecolor{myblue}{rgb}{0.9,0.9,0.98}

  \newcommand*\mybluebox[1]{%
    \colorbox{myblue}{\hspace{1em}#1\hspace{1em}}}

\allowdisplaybreaks 

\usepackage{relsize}

\begin{document}

\setlength{\abovedisplayskip}{8pt}
\setlength{\belowdisplayskip}{8pt}	
\newsavebox\myboxA
\newsavebox\myboxB
\newlength\mylenA

\newcommand*\xoverline[2][0.75]{%
    \sbox{\myboxA}{$#2$}%
    \setbox\myboxB\null
    \ht\myboxB=\ht\myboxA%
    \dp\myboxB=\dp\myboxA%
    \wd\myboxB=#1\wd\myboxA
    \sbox\myboxB{$\overline{\copy\myboxB}$}
    \setlength\mylenA{\the\wd\myboxA}
    \addtolength\mylenA{-\the\wd\myboxB}%
    \ifdim\wd\myboxB<\wd\myboxA%
       \rlap{\hskip 0.5\mylenA\usebox\myboxB}{\usebox\myboxA}%
    \else
        \hskip -0.5\mylenA\rlap{\usebox\myboxA}{\hskip 0.5\mylenA\usebox\myboxB}%
    \fi}
\makeatother

\makeatletter
\renewcommand*\env@matrix[1][\arraystretch]{%
  \edef\arraystretch{#1}%
  \hskip -\arraycolsep
  \let\@ifnextchar\new@ifnextchar
  \array{*\c@MaxMatrixCols c}}
\makeatother

\providecommand{\wbar}{\xoverline[0.9]{w}}
\providecommand{\ubar}{\xoverline{u}}

\newcommand{\nn}[1]{\ensuremath{\textstyle\mathsmaller{({#1})}}}
\newcommand{\crefpart}[2]{%
  \hyperref[#2]{\namecref{#1}~\labelcref*{#1}~\ref*{#2}}%
}
\newcommand\bigzero{\makebox(0,0){\text{\LARGE0}}}
	

%

\author{
Sedi Bartz\thanks{
Mathematics and Statistics, UMass Lowell, MA 01854, USA. E-mail:
\texttt{sedi\_bartz@uml.edu}.},~
Heinz H.\ Bauschke\thanks{
Mathematics, University
of British Columbia,
Kelowna, B.C.\ V1V~1V7, Canada. E-mail:
\texttt{heinz.bauschke@ubc.ca}.},~
Yuan Gao\thanks{
Mathematics, University
of British Columbia,
Kelowna, B.C.\ V1V~1V7, Canada. 
 E-mail: \texttt{y.gao@ubc.ca}.}~~~and~
 Walaa M. Moursi\thanks{Combinatorics and Optimization, University of Waterloo, Waterloo, Ontario\ N2L~3G1, Canada. 
 E-mail: \texttt{walaa.moursi@uwaterloo.ca}.}
}

\title{\textsf{ 
Strong Convergence of FISTA for Affinely Constrained 
Convex Quadratic Minimization
}
}

\date{April 2, 2026}

\maketitle

\begin{abstract}
In October 2025, research by Bo\c{t}, Fadili, and Nguyen, and by Jang and Ryu, led to the seminal result that 
Beck and Teboulle's FISTA converges \emph{weakly} to a 
minimizer of the sum of two convex functions resolving a long-standing open problem. The first \emph{strong} 
convergence result was obtained in November 2025 
by Moursi, Naguib, Pavlovic, and
Vavasis for affinely constrained 
convex minimization provided certain closedness conditions 
hold. 

In this paper, we prove strong convergence in the 
affine-quadratic case without any
closedness assumption. 
Specializing this to the unconstrained case, 
we obtain 
the strong convergence of 
Nesterov's accelerated gradient method when applied to a convex quadratic objective function. 
\end{abstract}
{ 


    \section{Introduction}

Throughout, we assume that 
\begin{subequations}\label{intro}
\begin{empheq}[box=\mybluebox]{equation}
\label{introX}
\text{$X$ is a real Hilbert space,} 
\end{empheq}
with inner product $\scal{\cdot}{\cdot}\colon X\times X\to\RR$, and induced norm $\|\cdot\|$.
We also assume that 
\begin{empheq}[box=\mybluebox]{align}
& f : X \to \mathbb{R} \ \text{is convex and $\beta$-smooth, 
where } \beta>0, \label{introf} \\
& g : X \to \left]-\infty, +\infty\right] \ \text{is convex, lower semicontinuous, and proper,} \label{introg} \\
& F := f + g, \label{introF}
\end{empheq}
that 
\begin{empheq}[box=\mybluebox]{equation}
\label{introS}
S :=  \Argmin F\neq\varnothing\quad\text{and}\qquad \mu := \min F(X),
\end{empheq}
and we denote the corresponding 
proximal gradient operator by 
\begin{empheq}[box=\mybluebox]{equation}
\label{introT}
T := \prox_{\frac{1}{\beta} g}
\circ \big(\Id - \tfrac{1}{\beta}\nabla f\big). 
\end{empheq}
\end{subequations}
Our interest lies in finding solutions to the problem 
\begin{equation}
\label{e:theprob}
\minimize{x\in X}\ F(x)=f(x)+g(x).
\end{equation}
One of the best algorithms for solving \cref{e:theprob} 
is Beck and Teboulle's famous FISTA. FISTA requires a 
\emph{parameter sequence} $(t_n)_\nnn$ that satisfies, 
for every $\nnn$, 
\begin{subequations}
\label{a:FISTA}
\begin{empheq}[box=\mybluebox]{align}
t_n&\geq \frac{n+2}{2}\geq 1=t_0, \label{e:t_n1}\\
t^2_n&\geq t^2_{n+1}-t_{n+1} \label{e:t_n2}.
\end{empheq}
Given any starting point $x_0\in X$, FISTA then generates two sequences 
$(x_n)_\nnn,(y_n)_\nnn$ in $X$ iteratively 
\begin{empheq}[box=\mybluebox]{align}
y_0 &:= x_0 \in X, \label{e:standard0}\\
x_{n+1}&:=Ty_n, \label{e:standard1}\\
y_{n+1}&:=x_{n+1}+\frac{t_n-1}{t_{n+1}}(x_{n+1}-x_n). 
\label{e:standard2}
\end{empheq}
\end{subequations}
In the smooth case, $g\equiv 0$, FISTA 
specializes to Nesterov's accelerated gradient descent \cite{Nest83}.
We shall refer to $(x_n,y_n)_\nnn$ as the combined FISTA sequence 
with starting point $x_0$ (for $(f,g)$, with smoothness constant 
$\beta$). 

The central convergence result for FISTA is
the following:

\begin{fact}[Nesterov 1983, Beck-Teboulle, 2009; Bo\c{t}-Fadili-Nguyen and Jang-Ryu, 2025] 
\label{f:weak}
Assume that \cref{intro} holds and 
let $(x_n)_\nnn$ and $(y_{n})_\nnn$ be the FISTA sequences 
generated by \cref{a:FISTA}. 
Then (see {\rm\cite{BT2009, Nest83}})
\begin{equation}
\label{e:value}
F(x_n)-\mu = \mathcal{O}\big(\tfrac{1}{n^{2}}\big)
\end{equation}
and 
(see {\rm\cite{BFN2025,JangRyu2025,BM2026}})
there exists $s\in S$ such that 
\begin{equation}
x_n\weakly s.
\end{equation}
\end{fact}

\begin{fact}[Moursi-Naguib-Pavlovic-Vavasis, 2025]
\label{f:Walaa}
Suppose that $Y$ is a real Hilbert space, 
$A\colon X\to Y$ is continuous, linear, and nonzero, 
$b\in Y$, 
$V$ is a closed affine subspace of $X$, 
$f\colon x\mapsto \thalb\|Ax-b\|^2$, 
and 
$g = \iota_V$.
Let $(x_n)_\nnn,(y_n)_\nnn$ be the FISTA sequences generated 
by \cref{a:FISTA}. 
Then 
\begin{equation}
	\label{e:Walaaconcw}
	x_n\weakly P_{S}x_0
\end{equation}
If, in addition, we have   
\begin{equation}
\label{e:Walaaneeds}
\ran A \text{ is closed and }
V-V + \ker A \text{ is closed,}
\end{equation}
then
\begin{equation}
	\label{e:Walaaconcst}
x_n\to P_{S}x_0.
\end{equation}
\end{fact}

The goal of this paper is to present a \emph{strong
convergence result for FISTA 
that does not require any closedness assumption such 
as \cref{e:Walaaneeds}.}
The resulting theorem (see \cref{t:ex}) 
appears to be new even 
in the classical case of Nesterov acceleration!

The remainder of the paper is organized as follows. 
In \cref{s:aux}, we collect some results that will 
make the proof of the main result more pleasant. 
The main results are then presented in 
\cref{s:main}. 
The final \cref{s:ex} focuses on
convex quadratic minimization with an affine constraint. 

Our notation is standard and follows largely \cite{BC2017} 
and \cite{Beck}.

\section{Auxiliary results}

In this section, we collect a few facts and results that 
will be used in the proofs of the main results. 

\label{s:aux}
    
\begin{fact}
\label{f:2}
Assume that \cref{intro} holds, and 
let $(x_n)_\nnn,(y_n)_\nnn$ be generated 
by \cref{a:FISTA}. 
Then the following hold: 
\begin{enumerate}
\item  
\label{f:2.4} 
$ \Fix T=S$. 
\item  
\label{f:2ne} 
$T$ is nonexpansive. 
\item 
\label{f:2.2}
$(\forall s\in S)(\forall\nnn)$
$\|x_n-s\|\leq\|x_0-s\|$. 
\item 
\label{f:2.3}
$x_n-y_n\to 0$. 
\end{enumerate}
\end{fact}
\begin{proof}
\cref{f:2.4}: 
This follows, e.g., 
from \cite[Proposition~26.1(iv)(a)]{BC2017}. 
\cref{f:2ne}: 
This follows, e.g., 
from \cite[Proposition~26.1(iv)(d)]{BC2017}. 
\cref{f:2.2}: See \cite[Theorem~3.4(i)]{MNPV}. 
\cref{f:2.3}: 
See \cite[Theorem~6.1]{BM2026}. 
\end{proof}

    \begin{lemma}
        \label{f:2.5}
        Let $(t_n)_\nnn$ satisfy \cref{e:t_n1} and \cref{e:t_n2}. Then
        \begin{align*}
            \frac{t_n-1}{t_{n+1}}\to 1. 
        \end{align*}
    \end{lemma}
    \begin{proof}
The proof is analogous to that of \cite[Lemma~4.3(ii)]{BBW}; however, 
the assumptions on the parameter sequence are slightly different 
in that paper which is why we include the details for completeness. 
    It follows from \cref{e:t_n1} that
\begin{equation}
\label{e:260316a}
t_n\to +\infty.
\end{equation}
Rewriting \cref{e:t_n2} and combining with \cref{e:260316a}, 
we obtain $t_n^2/t_{n+1}^2\geq 1-1/t_{n+1}\to 1$. 
The positivity of the parameter sequence now yields 
$\varliminf_{n\to\infty} {t_n/t_{n+1}} \geq 1$. Recalling \cref{e:260316a} again, we obtain 
\begin{align}
\label{e:260316c}
\varliminf_{n\to \infty}\frac{t_n-1}{t_{n+1}}\geq 1. 
\end{align}

We now show by induction that 
\begin{equation}
\label{e:260316b}
(\forall \nnn)\quad 
t_n\leq \frac{n+1+\sqrt{n+1}}{2}. 
\end{equation}
When $n=0$, then \cref{e:260316b} states $1\leq 1$. 
Now assume that \cref{e:260316b} holds for some $\nnn$. 
Combining \cref{e:t_n2} with the inductive hypothesis, we deduce 
\begin{equation}
\label{e:l3.11}
 t_{n+1}\leq \frac{1+\sqrt{1+4t_n^2}}{2}\leq \frac{1+\sqrt{1+(n+1+\sqrt{n+1})^2}}{2}. 
\end{equation}
On the other hand, because 
$
   \big(n+1+\sqrt{n+2}\big)^2-\big(1+(n+1+\sqrt{n+1})^2\big)
   =2(n+1)\big(\sqrt{n+2}-\sqrt{n+1}\big)>0,
$
we obtain 
\begin{equation}
\label{e:l3.12}
\sqrt{1+(n+1+\sqrt{n+1})^2}<n+1+\sqrt{n+2}. 
\end{equation}
Combining \cref{e:l3.11} and \cref{e:l3.12}, we get 
 $t_{n+1}<(n+2+\sqrt{n+2})/2$, 
which completes the inductive proof of \cref{e:260316b}.
Next, \cref{e:260316b} and \cref{e:t_n1} yield
\begin{equation}
t_n-1 \leq \frac{n-1+\sqrt{n+1}}{2}
\;\;\text{and}\;\;
\frac{1}{t_{n+1}} \leq \frac{2}{n+3}.
\end{equation}
Therefore, 
\begin{align}
\frac{t_n-1}{t_{n+1}}\leq \frac{n-1+\sqrt{n+1}}{n+3} \to 1
\end{align}
and so 
\begin{align}
\label{e:260316d}
\varlimsup_{n\to \infty}\frac{t_n-1}{t_{n+1}}\leq 1.
\end{align}
Combining \cref{e:260316c} with \cref{e:260316d} yields the result.
\end{proof}

\begin{lemma}
\label{l:affine}
Assume that \cref{intro} holds, and 
let $(x_n)_\nnn,(y_n)_\nnn$ be generated 
by \cref{a:FISTA}. 
Then the following hold:
\begin{enumerate}
\item 
\label{l:affine1}
$x_n-x_{n+1}\to 0$.
\item 
\label{l:affine2}
$y_n-y_{n+1}\to 0$.
\item 
\label{l:affine3}
$x_n-Tx_n\to 0$. 
\end{enumerate}
\end{lemma}
\begin{proof}
Let $n\geq 1$. 
\cref{l:affine1}: 
By \cref{e:t_n1}, $t_n\geq 3/2$ and so $t_n-1\geq 1/2>0$. 
By \cref{e:standard2}, 
$y_{n+1}-x_{n+1} = (t_n-1)/t_{n+1}(x_{n+1}-x_n)$. 
Thus
$x_{n+1}-x_n = t_{n+1}/(t_n-1)(y_{n+1}-x_{n+1})$ is well-defined. 
Taking the norm, and recalling \cref{f:2.5} and \cref{f:2}\cref{f:2.3}, 
we get 
\begin{equation}
0 \leq \|x_{n+1}-x_n\| = \frac{t_{n+1}}{t_n-1}\|y_{n+1}-x_{n+1}\| 
\to 1\cdot 0 = 0. 
\end{equation}
Hence $x_{n+1}-x_n\to 0$, as claimed. 

\cref{l:affine2}: 
Because 
$y_n-y_{n+1} = (y_n-x_n)+(x_n-x_{n+1}) + (x_{n+1}-y_{n+1})$, 
the conclusion follows by combining \cref{f:2}\cref{f:2.3} with 
\cref{l:affine1}. 

\cref{l:affine3}: 
Indeed, we estimate
\begin{align*}
\|x_{n+1}-Tx_{n+1}\|&=\|Ty_n-Tx_{n+1}\| \tag{by \cref{e:standard1}} \\
&=\left\|Ty_n-T\left(y_{n+1}-\frac{t_n-1}{t_{n+1}}(x_{n+1}-x_n)\right)\right\|\tag{by \cref{e:standard2}}\\
&\leq \left\|y_n-y_{n+1}+\frac{t_n-1}{t_{n+1}}(x_{n+1}-x_n)\right\| 
\tag{by \cref{f:2}\cref{f:2ne}}\\
&\leq \|y_n-y_{n+1}\|+\frac{t_n-1}{t_{n+1}}\|x_{n+1}-x_n\|
\tag{by the triangle inequality and \cref{e:t_n1}}\\
&\to 0 + 1\cdot 0 =0, \tag{by \cref{l:affine1}, \cref{l:affine2}, and \cref{f:2.5}}
\end{align*}
which yields $x_n-Tx_n\to 0$. 
\end{proof}

\begin{fact}
\label{f:2.6}
Let $L\colon X \to X$ be linear and nonexpansive, 
and let $q\in X$. 
Set $T:=L+q$ and suppose that $\Fix T\neq \varnothing$. 
Then there exists $a\in \Fix T$ such that 
the following hold: 
\begin{enumerate}
\item 
\label{f:2.6i} 
$(\Fix L)^\perp=\cran (\Id-L)$. 
\item 
\label{f:2.6ii}
$\Fix T=a+\Fix L$;
\item 
\label{f:2.6iii}
$(\forall x\in X)\ P_{\Fix T}x=a+P_{\Fix L}(x-a)$. 
\end{enumerate}
\end{fact}
\begin{proof}
\cref{f:2.6i}: This is implicit in \cite{RN} 
where it is shown that $\Fix L = \Fix L^*$ 
(see also \cite[Section~144]{RNbook} and 
\cite[Proposition~2.1]{BBG}). 
\cref{f:2.6ii}\&\cref{f:2.6iii}: See \cite[Lemma~3.2(i)\&(ii)]{BLM2017}. 
\end{proof}

\section{Abstract Main Results: Linear and Affine}

\label{s:main}

We are now ready for our two abstract main results. 
First, we obtain the linear version: 

\begin{theorem}[Linear Main Result]
\label{t:linear}
Assume that \cref{intro} holds, and 
let $(x_n)_\nnn,(y_n)_\nnn$ be generated 
by \cref{a:FISTA}. 
Furthermore, suppose that $T$ is \emph{linear}. 
Then 
\begin{align}
x_n\to P_{\Fix T}x_0
\;\;\text{and}\;\;
y_n\to P_{\Fix T}x_0. 
\end{align}
\end{theorem}
\begin{proof}
Let $\varepsilon>0$. 
Because $x_0-P_{\Fix T}x_0 \in (\Fix T)^\perp$, we learn from 
\cref{f:2.6}\cref{f:2.6i} that there exists 
$u_0\in X$ such that 
\begin{equation}
\label{e:linear1}
\|x_0-P_{\Fix T}x_0-(\Id-T)u_0\|<\varepsilon. 
\end{equation}
Now suppose that 
\begin{equation}
\label{e:260316e}
\text{$(u_n,v_n)_\nnn$ is the combined FISTA sequence
with starting point $u_0$.}
\end{equation}
Next, for every $\nnn$, we set 
\begin{subequations}
\begin{align}
d_n&:=x_n-P_{\Fix T}x_0-(\Id-T)u_n,\label{e:dn}\\
e_n&:=y_n-P_{\Fix T}x_0-(\Id-T)v_n. \label{e:en}
\end{align}
\end{subequations}

\textbf{Claim~1:} $(d_n,e_n)_\nnn$ is the combined FISTA 
sequence with starting point $d_0$.\\ 
We prove \textbf{Claim~1} by induction. 
Because $v_0 =u_0$ by definition, we obtain 
\begin{align*}
d_0 =x_0-P_{\Fix T}x_0-(\Id-T)u_0
=y_0-P_{\Fix T} x_0-(\Id-T)v_0
= e_0, 
\end{align*}
which verifies the base case. 
Let $\nnn$ and suppose we have now 
established the identity up to this index $n$. 
Then 
\begin{align}
d_{n+1}
&= x_{n+1}-P_{\Fix T}x_0-(\Id-T)u_{n+1} \tag{by \cref{e:dn}}\\
&= Ty_{n}-P_{\Fix T}x_0-(\Id-T)Tv_{n} \tag{by \cref{e:standard1} and 
\cref{e:260316e}}\\
&= T\big(y_{n}-P_{\Fix T}x_0-(\Id-T)v_{n}\big) 
\tag{by the linearity of $T$}\\
&= Te_n \tag{by \cref{e:en}}
\end{align}
and 
\begin{align}
e_{n+1}
&= y_{n+1}-P_{\Fix T}x_0-(\Id-T)v_{n+1}  \tag{by definition}\\
&= x_{n+1}+\tfrac{t_n-1}{t_{n+1}}(x_{n+1}-x_n)-P_{\Fix T}x_0-(\Id-T)\big(u_{n+1}+\tfrac{t_n-1}{t_{n+1}}(u_{n+1}-u_n)\big) 
\tag{by \cref{e:standard2} and  \cref{e:260316e}}\\
&= x_{n+1}-P_{\Fix T}x_0-(\Id-T)u_{n+1}+\tfrac{t_n-1}{t_{n+1}}\big[\big(x_{n+1}-(\Id-T)u_{n+1}\big)-\big(x_n-(\Id-T)u_n\big)\big] \tag{by the linearity of $T$}\\
&= d_{n+1}+\tfrac{t_n-1}{t_{n+1}}(d_{n+1}-d_n), \tag{by \cref{e:dn}}
\end{align}
which completes the proof of \textbf{Claim~1}.
      
      \textbf{Claim~2:} $(\forall \nnn)$ $\|d_n\|\leq \|d_0\|$.\\
Because $T$ is linear, we have that $0\in\Fix T$. 
By \cref{f:2}\cref{f:2.4}, $0\in S$. 
Hence \textbf{Claim~2} follows from \textbf{Claim~1} combined with \cref{f:2}\cref{f:2.2} (applied to $(d_n,e_n)_\nnn$). 

We now estimate
\begin{align}
\|x_n-P_{\Fix T}x_0\|
&= 
\|d_n+(\Id-T)u_n\|
\tag{by \cref{e:dn}}\\
&\leq 
\|d_n\|+\|(\Id-T)u_n\|
\tag{by the triangle inequality}\\
&\leq 
\|d_0\|+\|(\Id-T)u_n\|
\tag{by \textbf{Claim~2}}\\
&\to \|d_0\|
\tag{by \cref{e:260316e} and \cref{l:affine}\cref{l:affine3}}\\
&< \varepsilon,
\tag{by \cref{e:dn} and \cref{e:linear1}}
\end{align}
which yields $\varlimsup_{n\to\infty}\|x_n-P_{\Fix T}x_0\|\leq \varepsilon$.
Because we started with an arbitary $\varepsilon>0$, 
we see that $x_n\to P_{\Fix T}x_0$. 
Finally, this and \cref{f:2}\cref{f:2.3} imply 
$y_n = x_n + (y_n-x_n)\to P_{\Fix T}x_0$.
\end{proof}

A perturbation technique allows us now to extend \cref{t:linear} 
from the linear to the more general affine setting: 

\begin{theorem}[Affine Main Result]
\label{t:affinenew}
Assume that \cref{intro} holds, and 
let $(x_n)_\nnn,(y_n)_\nnn$ be generated 
by \cref{a:FISTA}. 
Furthermore, suppose that $T$ is \emph{affine}. 
Then
\begin{align}
  x_n\to P_Sx_0
\quad\text{and}\quad
  y_n\to P_Sx_0
\end{align}
\end{theorem}
\begin{proof}
By \cref{f:2}\cref{f:2ne}, $T$ is nonexpansive. 
By assumption and \cref{f:2}\cref{f:2.4},   
$\Fix T = S\neq\varnothing$. 
Because $T$ is affine, we may and do write 
\begin{equation}
\label{e:260317c}
T\colon X\to X\colon x\mapsto Lx+q,
\quad\text{where $L\colon X\to X$ is linear and nonexpansive, and $q\in X$.}
\end{equation}
In view of \cref{f:2.6}\cref{f:2.6ii}, 
there exists $a\in X $ such that 
\begin{equation}
\label{e:260317d}
a = Ta = La+q
\quad\text{and}\quad 
S = \Fix T = a+ \Fix L.
\end{equation}
Now set 
\begin{equation}
\widetilde{f}\colon X\to \RR\colon 
x\mapsto f(x) + \beta\scal{x}{q}.
\end{equation}
Then $\nabla\widetilde{f} = \nabla f(\cdot)+\beta q$, 
which shows that 
$\widetilde{f}$ is $\beta$-smooth with
$\tfrac{1}{\beta}\nabla\widetilde{f} = 
\tfrac{1}{\beta}\nabla f(\cdot)+q$, and 
\begin{equation}
\label{e:260317a}
(\forall x\in X)\quad
x-\tfrac{1}{\beta}\nabla\widetilde{f}(x) = 
x- \tfrac{1}{\beta}\nabla f(x)-q.
\end{equation}
Next, we set 
\begin{equation}
\widetilde{g}\colon X\to \RX\colon
x\mapsto g(x+q).
\end{equation}
Let $x\in X$. 
By \cite[Proposition~24.8(ii)]{BC2017}, 
\begin{equation}
\label{e:260317b}
P_{\frac{1}{\beta}\widetilde{g}}(x)
= -q + P_{\frac{1}{\beta}g}(x+q).
\end{equation}
Altogether, 
\begin{align}
P_{\frac{1}{\beta}\widetilde{g}}\big( 
x - \tfrac{1}{\beta}\nabla\widetilde{f}(x)
\big) 
&=
P_{\frac{1}{\beta}\widetilde{g}}\big(
x- \tfrac{1}{\beta}\nabla f(x)-q
\big)
\tag{by \cref{e:260317a}}\\
&=
-q + P_{\frac{1}{\beta}g}\big(
(x- \tfrac{1}{\beta}\nabla f(x)-q)+q
\big)
\tag{by \cref{e:260317b}}\\
&=-q+Tx\tag{by the definition of $T$}\\
&= Lx,
\tag{by \cref{e:260317c}}
\end{align}
which reveals $L$ as the prox-grad operator 
of $(\widetilde{f},\widetilde{g})$, with smoothness parameter 
$\beta$. 

For every $\nnn$, we have 
\begin{align}
x_{n+1}-a 
&= Ty_n-a \tag{by assumption}\\
&= Ly_n+q-a \tag{by \cref{e:260317c}}\\
&= Ly_n-La \tag{by \cref{e:260317d}}\\
&= L(y_n-a) \notag
\end{align}
and 
\begin{align}
y_{n+1}-a 
&= 
(x_{n+1}-a)+\frac{t_n-1}{t_{n+1}}\big((x_{n+1}-a)-(x_n-a)\big).
\end{align}
Thus, 
$(x_n-a,y_n-a)_\nnn$ is the 
combined FISTA sequence of $(\widetilde{f},\widetilde{g})$
with smoothness parameter $\beta$ and starting point $x_0-a$. 
\cref{t:linear} now yields
\begin{equation}
x_n-a \to P_{\Fix L}(x_0-a) \leftarrow y_n-a.
\end{equation}
Therefore, 
using also \cref{f:2.6}\cref{f:2.6iii} and \cref{f:2}\cref{f:2.4}, we obtain
\begin{equation}
\lim_{n\to\infty} x_n 
= \lim_{n\to \infty} y_n 
= a+P_{\Fix L}(x_0-a)
= P_{\Fix T}x_0 = P_Sx_0,
\end{equation}
and we're done. 
\end{proof}

    \section{Quadratic Minimization with an 
Affine Constraint}

\label{s:ex}
    
In this section, we apply the affine main result (\cref{t:affinenew}) 
to a specific setting, where we assume that 
\begin{subequations}
\label{last}
\begin{empheq}[box=\mybluebox]{align}
& Y \text{ is a real Hilbert space,} \label{lastY}\\
& A \colon X\to Y \text{ is linear, continuous, and nonzero,} 
\label{lastA}\\
& b\in Y, \label{lastb}\\
& f \colon x\mapsto \thalb\|Ax-b\|^2, \label{lastf}\\
& \beta \geq \|A^*A\|, \label{lastbeta}\\
& V \text{ is a closed affine subspace of $X$,} \label{lastV}\\
& g = \iota_V \label{lastg}.
\end{empheq}
\end{subequations}
Then our original problem \cref{e:theprob} of 
minimizing $F = f+g$ amounts to 
the problem 
\begin{equation}
\label{e:ex}
\begin{split}
   \text{minimize } &\tfrac{1}{2}\|Ax-b\|^2\\
   \text{subject to } &x\in V, 
\end{split}
\end{equation}
where the objective function 
$x\mapsto \thalb\|Ax-b\|^2$ is convex and quadratic, the 
constraint set $V$ is affine, and the set of solutions $S$ is assumed
to be nonempty. 

We are now ready for the following strong convergence result
\emph{which appears to be new even in the unconstrained ($V=X$) case}: 

\begin{theorem}[affine-quadratic minimization]
\label{t:ex}
Assume that \cref{intro} holds, and 
let $(x_n)_\nnn,(y_n)_\nnn$ be generated 
by \cref{a:FISTA}. 
Furthermore, suppose that \cref{last} holds. 
Consider the optimization problem \cref{e:ex}, and set 
$v_0:=P_V0$, and $\Par V:=V-v_0$. 
Then the proximal gradient operator is given by 
\begin{equation}
\label{e:260319a}
T\colon x\mapsto P_{\Par V}(\Id-\tfrac{1}{\beta}A^*A)x+\tfrac{1}{\beta}P_{\Par V}A^*b+v_0,
\end{equation}
and 
\begin{align}
\label{e:newstrong}
x_n\to P_{S}x_0
\;\;\text{and}\;\;
y_n\to P_{S}x_0. 
\end{align}
where $S$ is the set of solutions of \cref{e:ex}. 
\end{theorem}
\begin{proof}
The formula \cref{e:260319a}, 
which was derived in \cite[Section~4]{MNPV}, makes it 
clear that $T$ is affine. 
The result thus follows from \cref{t:affinenew}. 
\end{proof}

\begin{remark}[comparison to \cite{MNPV}]
Recall the setting of \cref{t:ex}.
Moursi et al.\ show (see \cref{f:Walaa} above and 
\cite[Theorem~5.4]{MNPV}) 
that $x_n\to P_Sx_0$ provided that 
\begin{subequations}
\begin{equation}
\label{e:walaa1}
F(x_n)\to \mu
\end{equation}
and 
\begin{equation}
\label{e:walaa2}
\ran A\; \text{ is closed, and }\;
\Par V + \ker A \text{ is closed}
\end{equation}
\end{subequations}
hold. 
Note that \cref{e:walaa1} is automatic 
in our present setting (see \cref{e:value}); however, 
Moursi et al.\ allow for potentially more general parameter
sequences. 
Indeed, any parameter sequence that ensures \cref{e:walaa1} holds will work.
While \cref{e:walaa2} covers already a good amount of examples 
because it is automatic in finite-dimensional settings  (see \cite[Section~7]{MNPV}), our present setting does not 
require these closedness assumptions in \cref{e:walaa2}. 
\end{remark}

\begin{corollary}[accelerated alternating affine projections]
\label{c:aaap}
Assume that \cref{intro} holds, and 
let $(x_n)_\nnn$, $(y_n)_\nnn$ be generated 
by \cref{a:FISTA}.
Furthermore, suppose that \cref{last} holds, 
with 
$f = \thalb d_U^2$, where $U$ is a
closed affine subspace of $X$ and where\footnote{The assumption $S\neq \varnothing$ is true if, for instance, $U \cap V \neq \varnothing$, in which case 
		$\fix P_VP_U =U\cap V$, or if $U+V$ 
		is closed. }
$S =\fix P_VP_U \neq\varnothing$.
Then $T = P_VP_U$ and
\begin{align}
x_n\to P_{\fix P_VP_U}x_0
	\;\;\text{and}\;\;
	y_n\to P_{\fix P_VP_U}x_0. 
\end{align}
Note that if $U\cap V\neq \varnothing$,
then $S=U\cap V$ and  
\begin{align}
x_n\to P_{U\cap V}x_0
\;\;\text{and}\;\;
y_n\to P_{U\cap V}x_0. 
\end{align}
\end{corollary}

\begin{corollary}[Nesterov's acceleration for a least squares
objective] 
\label{c:Nesquad}
Assume that \cref{intro} holds, and 
let $(x_n)_\nnn$, $(y_n)_\nnn$ be generated 
by \cref{a:FISTA}.
Furthermore, suppose that \cref{last} holds, 
with $V=X$. Then 
$T \colon x \mapsto  x-\tfrac{1}{\beta}A^*(Ax-b)$ and 
\begin{align}
x_n\to P_{S}x_0
\;\;\text{and}\;\;
y_n\to P_{S}x_0, 
\end{align}
where $S = \Argmin f \neq\varnothing$ is assumed to be nonempty. 
\end{corollary}

\begin{remark}[Nesterov's acceleration for a general 
convex quadratic objective]
Suppose (momentarily) that $f$ is given by $x\mapsto 
\thalb\scal{x}{Ax} + \scal{x}{b}$, 
where $A=A^*$ and $A$ is monotone. 
Then $\nabla f(x) = Ax + b$ is $\beta$-Lipschitz when 
$\beta\geq\|A\|$ and strong convergence follows again from 
\cref{t:affinenew}
because $T \colon x\mapsto x- \tfrac{1}{\beta}Ax-\tfrac{1}{\beta}b$ is clearly affine. 
\end{remark}

We conclude by presenting three scenarios where the 
strong convergence result \cref{e:newstrong} is
a truly novel conclusion. In each of these examples, 
$(x_n,y_n)_\nnn$ is the combined FISTA sequence 
with starting point $x_0\in X$ chosen arbitrarily.

\begin{example}[accelerated alternating linear projections]
Suppose that $X=\ell^2(\{1,2,\ldots\})$ 
with the usual Schauder basis $(e_n)_{n\geq 1}$, 
$Y = X$, $b=0$, and
set 
\begin{align}
U&:=\cspan\{e_1,e_3,e_5,\ldots\}.
\end{align}
Suppose that 
\begin{subequations}
\begin{align}
A&=\Id-P_U = P_{U^\perp},\\
V&=\cspan\{ \cos(\gamma_n)e_{2n-1}+\sin(\gamma_n)e_{2n}\ |\ n\in \{1, 2, \ldots\}\},
\end{align}
\end{subequations}
and 
let $(\gamma_n)_\nnn$ be a sequence in $]0, \tfrac{\pi}{2}[$ 
that decreases to $0$. 
The prox-grad operator $T$ coincides with the 
alternating projections operator $P_VP_U$ in this setting. 
While $\ran(A) = U^\perp$ is closed, 
the Minkowski sum $\Par V + \ker A = V+U$ is not closed 
(by \cite[Example~4.1(iv)]{BE2005}); thus, 
\cref{e:walaa2} fails.
Moreover, 
\begin{align}
S=\Argmin (f+g)=\Argmin \big(\tfrac{1}{2}d^2_U+\iota_{V}\big)=U\cap V=\{0\}
\end{align}
and so $P_S\equiv 0$. 
Therefore, by \cref{t:ex}, $x_n\to 0$ and $y_n\to 0$.
\end{example}

\begin{example}
Suppose that $X=\ell^2(\NN)$, 
with the usual Schauder basis $(e_n)_{\nnn}$, 
$Y=X$, and let $R\colon X\to X$ be the right-shift operator. 
Suppose that $A = \Id-R$, $b=0$, and $V=\RR\cdot e_0$. 
This time, $\ran A$ is not closed\footnote{
Indeed, note first that $\ker(\Id-R) = \{0\}^\perp = X$, so 
for every $n\geq 1$, 
$z_n := (e_1+\cdots+e_{n})/\sqrt{n} \in X$ with 
$\|z_n\|=1$. On the other hand, 
$\|(\Id-R)z_n\| = \|e_1-e_{n+1}\|/\sqrt{n} = 
\sqrt{2/n}\to 0$. 
By \cite[Fact~2.26]{BC2017}, $\ran(\Id-R) = \ran(A)$ 
is not closed.}, 
so \cref{e:walaa2} fails. 
Note that the unique minimizer of 
$f(x) = \thalb\|x-Rx\|^2$ is $0$, which lies already in $V$. 
Therefore, by \cref{t:ex}, $x_n\to 0$ and $y_n\to 0$.
\end{example}

\begin{example}
Suppose that $X=\ell^2(\NN)$, 
with the usual Schauder basis $(e_n)_{\nnn}$, 
$Y=V=X$, and suppose that 
$A = \Diag(\gamma_n)_\nnn\colon 
(x_n)_\nnn\mapsto (\gamma_nx_n)_\nnn$, where 
$(\gamma_n)_\nnn$ is a fixed sequence of positive 
real numbers with $\gamma_n\to 0$, and $b=0$.
Then $\ran A$ is not closed\footnote{
Indeed, 
$\ker A = \{0\}$. So $e_n \in (\ker A)^\perp$ 
and $\|e_n\|=1$ 
for all $\nnn$. 
Moreover, $\|Ae_n\| = \gamma_n \to 0$. 
By \cite[Fact~2.26]{BC2017}, $\ran(A)$ is not closed.} 
so \cref{e:walaa2} fails. 
The unique unconstrained minimizer of $f(x)=\thalb\|Ax\|^2$ is 
$0$; thus, 
by \cref{t:ex}, $x_n\to 0$ and $y_n\to 0$.
\end{example}

\section*{Acknowledgments}
The research of SB was partially supported by a collaboration grant for mathematicians of the Simons Foundation. 
The research of HHB and WMM was partially supported 
by Discovery Grants 
of the Natural Sciences and Engineering Research Council of
Canada. 
The research of WMM is also partially supported 
by the Ontario Early Researcher Award.

\end{document}